\nonstopmode
\documentclass[10pt]{amsart}
\usepackage{graphicx}
\usepackage{latexsym}
\usepackage{fancyhdr}
\usepackage{amsmath, amssymb, amsthm}
\usepackage[all]{xy}
\usepackage{pdflscape}
\usepackage{longtable}
\usepackage{rotating}
\usepackage{verbatim}
\usepackage{hyperref}
\hypersetup{
  linkcolor  = blue,
  citecolor  = blue,
  urlcolor   = blue,
  colorlinks = true,
}

\usepackage{cleveref}
\usepackage{subfigure}
\usepackage{mathrsfs}
\usepackage{mdwlist}
\usepackage{dsfont}
\usepackage{mathtools}
\usepackage{float}
\usepackage{color}
\usepackage{stmaryrd}

\usepackage{tkz-base}
\usepackage{pgfplots}
\pgfplotsset{compat=newest}

\usepackage{tkz-euclide}
\usepackage{etoolbox}

\definecolor{teal}{rgb}{0.0, 0.5, 0.5}

\newcounter{mnotecount}[section]

\newcommand{\rmnote}[1]{}

\allowdisplaybreaks

\DeclareFontFamily{U}{mathb}{\hyphenchar\font45}
\DeclareFontShape{U}{mathb}{m}{n}{
      <5> <6> <7> <8> <9> <10> gen * mathb
      <10.95> mathb10 <12> <14.4> <17.28> <20.74> <24.88> mathb12
      }{}
\DeclareSymbolFont{mathb}{U}{mathb}{m}{n}
\DeclareFontSubstitution{U}{mathb}{m}{n}

\let\dot\relax
\DeclareMathAccent{\dot}{0}{mathb}{"39}
\let\ddot\relax
\DeclareMathAccent{\ddot}{0}{mathb}{"3A}
\let\dddot\relax
\DeclareMathAccent{\dddot}{0}{mathb}{"3B}
\let\ddddot\relax
\DeclareMathAccent{\ddddot}{0}{mathb}{"3C}

\theoremstyle{plain}
\newtheorem*{theorem*}{Theorem}
\newtheorem{theorem}{Theorem}[section]
\newtheorem*{lemma*}{Lemma}
\newtheorem{lemma}[theorem]{Lemma}
\newtheorem*{assumption*}{Assumption}

\newtheorem*{proposition*}{Proposition}
\newtheorem{proposition}[theorem]{Proposition}
\newtheorem*{corollary*}{Corollary}
\newtheorem{corollary}[theorem]{Corollary}
\newtheorem*{claim*}{Claim}
\newtheorem{claim}{Claim}

\newtheorem*{conjecture*}{Conjecture}

\newtheorem*{question*}{Question}
\newtheorem*{result*}{Result}

\theoremstyle{definition}
\newtheorem*{definition*}{Definition}

\newtheorem*{example*}{Example}

\newtheorem*{algorithm*}{Algorithm}
\newtheorem*{remark*}{Remark}
\newtheorem*{remarks*}{Remarks}
\newtheorem{remark}[theorem]{Remark}

\newtheorem*{convention*}{Convention}

\theoremstyle{plain}
\newtheorem{maintheorem}{Theorem}

\numberwithin{equation}{section}

\def\al{\alpha}

\def\la{\lambda}

\def\rh{\rho}

\def\si{\sigma}

\def\vh{\varphi}

\def\om{\omega}

\def\Ga{\Gamma}

\def\N{\mathbb{N}}

\def\R{\mathbb{R}}

\def\cA{\mathcal{A}}

\def\cH{\mathcal{H}}

\def\cM{\mathcal{M}}
\def\cN{\mathcal{N}}

\def\cQ{\mathcal{Q}}

\def\fR{\mathfrak{R}}

\def\sS{\mathscr{S}}

\def\p{\partial}

\def\<{\langle}
\def\>{\rangle}

\def\ol{\overline}

\let\on=\operatorname

\newcommand{\sr}[1]%
{\ifmmode{}^\dagger\else${}^\dagger$\fi\ifvmode
\vbox to 0pt{\vss
 \hbox to 0pt{\hskip\hsize\hskip1em
 \vbox{\hsize3cm\raggedright\pretolerance10000
 \noindent #1\hfill}\hss}\vss}\else
 \vadjust{\vbox to0pt{\vss%
 \hbox to 0pt{\hskip\hsize\hskip1em%
 \vbox{\hsize3cm\raggedright\pretolerance10000%
 \noindent #1\hfill}\hss}\vss}}\fi%
}

\providecommand{\mapsfrom}{\kern.2em%
\setbox0=\hbox{$\leftarrow$\kern-.10em\rule[0.26mm]{0.1mm}{1.3mm}}\box0%
\kern.3em}

\title[Definable Lipschitz selections]
{Definable Lipschitz selections for affine-set valued maps}

\author[Adam Parusi\'nski and  Armin Rainer]
{Adam Parusi\'nski and Armin Rainer}

\address {Adam Parusi\'nski: Universit\'e C\^ote d'Azur,  CNRS,  LJAD, UMR 7351, 06108 Nice, France}
\email{adam.parusinski@univ-cotedazur.fr}

\address{Armin Rainer:\footnote{Current address: Faculty of Mathematics and Geoinformation, Institute for Statistics and Mathematical Methods in Economics, TU Wien, Wiedner Hauptstraße 8, 1040 Austria. Email address: armin.rainer@tuwien.ac.at} 
Fakult\"at f\"ur Mathematik, Universit\"at Wien,
Oskar-Morgenstern-Platz~1, A-1090 Wien, Austria}
\email{armin.rainer@univie.ac.at}

\begin{document}

\begin{abstract}
    Whitney's extension problem, i.e., how one can tell whether a function 
    $f : X \to \R$, $X \subseteq \R^n$, is the restriction of a $C^m$-function on $\R^n$, 
    was solved in full generality by Charles Fefferman in 2006.
    In this paper, we settle the $C^{1,\om}$-case of a related conjecture: 
    given that $f$ is semialgebraic and $\om$ is a semialgebraic modulus of continuity,   
    if $f$ is the restriction of 
    a $C^{1,\om}$-function then it is the restriction of a semialgebraic $C^{1,\om}$-function.
    We work in the more general setting of sets that are definable in an 
    o-minimial expansion of the real field.   
    An ingenious argument of Brudnyi and Shvartsman 
    relates the existence of $C^{1,\om}$-extensions to the existence of Lipschitz 
    selections of certain affine-set valued maps. 
    We show that if a definable affine-set valued map has Lipschitz selections then it also 
    has definable Lipschitz selections. In particular, we obtain a Lipschitz 
    solution (more generally, $\om$-H\"older solution, 
    for any definable modulus of continuity $\om$) of the 
    definable Brenner--Epstein--Hochster--Koll\'ar problem.
    In most of our results we have control over the respective (semi)norms.
\end{abstract}

\thanks{Supported by FWF-Project P 32905-N and Oberwolfach Research Fellows (OWRF) ID 2244p}
\keywords{O-minimal structures, Whitney's extension problem, linear systems, Lipschitz selections}
\subjclass[2020]{
    03C64,  	    
    14P10,      
	26B35,  	
    26E25,      
    32B20,  	
46E15}      
\date{\today}

\maketitle

\section{Introduction}

In this paper, we settle the $C^{1,\om}$-case of an open problem raised by Bierstone and Milman (cf.\ \cite{Zobin:aa}): 
\emph{Given a compact semialgebraic subset $X \subseteq \R^n$ and a semialgebraic function $f : X \to \R$ 
which is the restriction of a $C^m$-function on $\R^n$, does there exist a semialgebraic $C^m$-extension of $f$ to $\R^n$?}

For general $m$, the answer to this question is known to be affirmative in dimension $n \le 2$,
due to Fefferman and Luli \cite{Fefferman:2022aa}. 
In arbitrary dimension, 
Aschenbrenner and Thamrongthanyalak \cite{Aschenbrenner:2019aa} 
proved the $C^1$-version of the statement.
A solution with loss of regularity is due to Bierstone, Campesato, and Milman 
\cite{Bierstone:2021aa}. 

The philosophy of our approach is related to Aschenbrenner and Thamrongthanyalak's  
who prove and use a definable version of Michael's selection theorem. 
``Definable'' means that the sets and maps belong to a fixed o-minimal expansion of the real field
(of which semialgebraic sets are a basic example). Michael's theorem concerns the existence of 
continuous selections of set-valued maps. 
Already the existence of definable Lipschitz selections is in general not known, since the classical results 
are based on transcendental methods (e.g.\ the Steiner point).

Nevertheless, for certain affine-set valued maps, that are relevant for the $C^{1,\om}$-Whitney extension problem, 
Lipschitz selections can be constructed in a way that preserves definability in the given o-minimal structure. 
This construction is due to Brudnyi and Shvartsman \cite{Brudnyi:2001aa}; see also the references therein 
for precursors.

Let us now describe our results in more detail.

\subsection{Main results}
Let an o-minimal expansion of the real field be fixed. 
Throughout the paper, a set $X \subseteq \R^n$ is called \emph{definable} if it is definable in this fixed o-minimal structure. 
A map $\vh : X \to \R^m$ is definable if its graph $\{(x,\vh(x)) : x \in X\}$ is a definable subset of $\R^n \times \R^m$. 
Cf.\ \Cref{ssec:o-minimal}.

By a \emph{modulus of continuity} we always mean a positive, continuous, increasing, and concave function $\om : (0,\infty) \to (0,\infty)$ 
such that $\om(t) \to 0$ as $t \to 0$. A modulus of continuity $\om$ is called \emph{definable} if the function $\om$ is definable.

By definition, $C^{1,\om}(\R^n)$ is the space of all $C^1$-functions $f : \R^n \to \R$ that are globally bounded on $\R^n$ 
and whose partial derivatives of first order are globally bounded and globally $\om$-H\"older on $\R^n$.
Equipped with its natural norm, this space is a Banach space. 
We write $C^{1,\om}_{\on{def}}(\R^n)$ for the subspace of definable functions in $C^{1,\om}(\R^n)$. 
Given a definable subset $X \subseteq \R^n$, 
we denote by $C^{1,\om}(\R^n)|_X$ and $C^{1,\om}_{\on{def}}(\R^n)|_X$ the respective trace spaces on $X$.
Let $\R^X_{\on{def}}$ be the set of all definable functions $f : X \to \R$.
For precise definitions we refer to \Cref{ssec:spaces}.

We will prove the following theorem.

\begin{maintheorem} \label{thm:A}
    Let $\om$ be a definable modulus of continuity. 
    Let $X \subseteq \R^n$ be a closed definable 
    set and $f : X \to \R$ a definable function.
    Then the following conditions are equivalent.
    \begin{enumerate}
        \item $f$ is the restriction of a $C^{1,\om}$-function on $\R^n$.
        \item $f$ is the restriction of a definable $C^{1,\om}$-function on $\R^n$.
    \end{enumerate}
    That means
    \begin{equation} \label{eq:altequiv}
        \R^X_{\on{def}} \cap C^{1,\om}(\R^n)|_X = C^{1,\om}_{\on{def}}(\R^n)|_X.
    \end{equation}
    Moreover, a subset of the set \eqref{eq:altequiv} is bounded in $C^{1,\om}(\R^n)|_X$
    if and only if it is bounded in $C^{1,\om}_{\on{def}}(\R^n)|_X$.

    In the Lipschitz case $\om(t) = t$, for compact definable $X$ we even have  
    \begin{equation} \label{eq:equivnorms}
        \|f\|_{C^{1,1}_{\on{def}}(\R^n)|_X} \approx \|f\|_{C^{1,1}(\R^n)|_X}.
    \end{equation}
\end{maintheorem}

In particular, \eqref{eq:altequiv} holds for all classical H\"older classes with rational H\"older exponent, 
since $\om(t) = t^\al$ with rational $\al \in (0,1]$ is semialgebraic.
We conjecture that
\begin{equation} \label{eq:equivnormsom}
    \|f\|_{C^{1,\om}_{\on{def}}(\R^n)|_X} \approx \|f\|_{C^{1,\om}(\R^n)|_X}
\end{equation}
for \emph{any} definable modulus of continuity $\om$.   

The meaning of ``$\approx$'' in \eqref{eq:equivnorms} and \eqref{eq:equivnormsom} is 
that either quotient of the two sides lies in the interval $[C^{-1},C]$ 
for some constant $C\ge1$ depending only on $n$.

The proof of \Cref{thm:A} is based on an ingenious argument of Brudnyi and Shvartsman \cite{Brudnyi:2001aa}
that relates the existence of $C^{1,\om}$-extensions to the existence of Lipschitz selections for 
certain maps which take affine subspaces of $\R^n$ as values. 

The next result, \Cref{thm:B}, states that, if these maps happen to be definable and admit a Lipschitz selection, then 
they have a definable Lipschitz selection.

We denote by $\cA_k(\R^n)$ the set of all affine subspaces 
of $\R^n$ of dimension at most $k$. A pseudometric space $(\cM,\rh)$ is said to be definable if $\cM$ is a definable subset 
of some $\R^N$ and the pseudometric $\rh : \cM \times \cM \to [0, \infty)$ is a definable function. 
A map $F : \cM \to \cA_k(\R^n)$ is called definable if its graph $\bigcup_{m \in \cM} (\{m\} \times F(m))$ is a 
definable subset of $\R^N \times \R^n$. 
A \emph{selection} of $F$ is a map $f : \cM \to \R^n$ such that $f(m) \in F(m)$ for all $m \in \cM$.
Cf.\ \Cref{ssec:pseudometric} and \Cref{ssec:setvalued}.

\begin{maintheorem} \label{thm:B}
    Let $(\cM,\rh)$ be a definable pseudometric space and $F : \cM \to \cA_k(\R^n)$ a definable map.   
    Then the following conditions are equivalent.
    \begin{enumerate}
        \item $F$ has a Lipschitz selection.
        \item $F$ has a definable Lipschitz selection.    
    \end{enumerate}
    Moreover, if there is a Lipschitz selection $f$ of $F$ with Lipschitz seminorm $|f|_{\on{Lip}(\cM,\R^n)}$, then 
    there is a definable Lipschitz selection $g$ of $F$ with
    \begin{equation} \label{eq:thmB}
        |g|_{\on{Lip}(\cM,\R^n)} \le C(k,n)\, |f|_{\on{Lip}(\cM,\R^n)},
    \end{equation}
    where $C(k,n)>0$ is a constant that only depends on $k$ and $n$.
\end{maintheorem}

To deduce \Cref{thm:A} from \Cref{thm:B} we actually need a slightly stronger version of the latter which is 
stated in \Cref{thm:BS23def}.

Another application of \Cref{thm:B} is an $\om$-H\"older (in particular, Lipschitz) 
solution of the definable Brenner--Epstein--Hochster--Koll\'ar problem.

\begin{maintheorem} \label{thm:C}
    Let $\om$ be a definable modulus of continuity. 
    Let $A_{ij},b_i : X \to \R$, for $i=1,\ldots,N$ and $j=1,\ldots,M$, be definable functions defined on a definable 
    subset $X \subseteq \R^n$.
    Consider the linear system of equations
    \begin{equation} \label{eq:system}
        \sum_{j=1}^M A_{ij} f_j = b_i, \quad i = 1,\ldots,N,
    \end{equation}
    in the unkowns $f_j$, $j = 1,\ldots,M$.
    Then the following conditions are equivalent.
    \begin{enumerate}
        \item  The system \eqref{eq:system} admits an $\om$-H\"older solution. 
        \item  The system \eqref{eq:system} admits a definable $\om$-H\"older solution. 
    \end{enumerate}
    Moreover, if there is an $\om$-H\"older solution $f=(f_1,\ldots,f_M)$ of \eqref{eq:system} with $\om$-H\"older seminorm 
    $|f|_{C^{0,\om}(X,\R^M)}$,
    then there is a definable $\om$-H\"older solution $g=(g_1,\ldots,g_M)$ of \eqref{eq:system}  such that 
    \begin{equation} \label{eq:thmC}
        |g|_{C^{0,\om}(X,\R^M)} \le C(M) \, |f|_{C^{0,\om}(X,\R^M)}.
    \end{equation}
\end{maintheorem}

A $C^0$-version of this theorem is due to Aschenbrenner and Thamrongthanyalak \cite{Aschenbrenner:2019aa};
for the semialgebraic setting see also Fefferman and Koll\'ar \cite{FeffermanKollar13}.

Let us emphasize that, due to \cite{Brudnyi:2001aa}, the respective first conditions in \Cref{thm:A} and \Cref{thm:B} 
are characterized by \emph{finiteness principles}:
\begin{itemize}
    \item An arbitrary function $f : X \to \R$, $X \subseteq \R^n$, belongs to
        $C^{1,\om}(\R^n)|_X$ if and only if for each subset $Y \subseteq X$ of cardinality $\# Y \le 3 \cdot 2^{n-1}$ 
        there is $F_Y \in C^{1,\om}(\R^n)$ such that $F_Y = f$ on $Y$ and $\sup_{Y}\|F_Y\|_{C^{1,\om}(\R^n)} < \infty$.
    \item An arbitrary map $F : \cM \to \cA_k(\R^n)$, where $\cM$ is any pseudometric space,
        admits a Lipschitz selection if and only if 
        for each subset $\cN \subseteq \cM$  of cardinality $\# \cN \le  2^{k+1}$  
        the restriction $F|_\cN$ has a Lipschitz selection $f_\cN$ such that $\sup_{\cN} |f_\cN|_{\on{Lip}(\cN,\R^n)} <\infty$.
\end{itemize}
As shown in \cite{Brudnyi:2001aa}, the cardinalities of the finite test sets cannot be reduced.
Obviously, we immediately get a corresponding finiteness principle for \Cref{thm:C}.

The proofs of the main results, \Cref{thm:A} and \Cref{thm:B}, are essentially a 
careful verification that the constructions of \cite{Brudnyi:2001aa} are done in a definable way.
We first prove \Cref{thm:B} in \Cref{sec:selection} and then deduce \Cref{thm:A} from it 
in \Cref{sec:extension}. \Cref{thm:C} is a simple consequence of \Cref{thm:B}; see \Cref{sec:proofC}.

For the transition from a definable Whitney jet of class $C^{1,\om}$ on $X$ to a definable $C^{1,\om}$-function on $\R^n$,
we use our recent paper \cite{ParusinskiRainer:2023aa} on the uniform extension of definable Whitney jets of class $C^{m,\om}$ 
(which is a variation of the definable $C^m$ Whitney extension theorem \cite{Kurdyka:1997ab, Kurdyka:2014aa, Thamrongthanyalak:2017aa}).
At this stage, we cannot control the norms in this extension 
and thus are not able to prove \eqref{eq:equivnormsom}, but we prove that the extension can be done in 
a bounded way; see also \Cref{ssec:furtherremarks}.
Alternatively, in the $C^{1,1}$-case one can use 
the explicit formula due to Azaga, Le Gruyer, and Mudarra \cite{Azagra:2018aa}
which is visibly definable and admits a control of the norms which leads to \eqref{eq:equivnorms}.
The explicit formula also yields an immediate proof of the definable Kirszbraun theorem; see \Cref{ssec:Kirszbraun}.

\subsection*{Notation}

We equip $\R^n$ with the maximum-norm $\|x\| := \max_{1 \le i \le n} |x_i|$. 
If a different norm is used, e.g., the Euclidean norm $\|x\|_2 := (\sum_{1 \le i \le n} x_i^2)^{1/2}$, 
then it will be explicitly stated. 
Note that $\|x\| \le \|x\|_2 \le \sqrt n\, \|x\|$. 
The closed $\| \cdot \|$-balls $Q(x,r) := \{y \in \R^n : \|x-y\|\le r\}$ are cubes with sides parallel to the 
coordinate axes. We also write $Q(r):= Q(0,r)$ so that $Q(x,r) = x+Q(r)$.
If $\la>0$ then $\la Q(x,r)$ denotes the cube $x + Q(\la r)$.
The standard scalar product in $\R^n$ is denoted by $\langle x,y \rangle := \sum_{i=1}^n x_iy_i$.

\section{Preliminaries} \label{sec:prelim}

In this section, we recall definitions and fix notation.

\subsection{O-minimal expansions of the real field} \label{ssec:o-minimal}

An \emph{o-minimal expansion of the ordered field of real numbers} 
is a family $\sS = (\sS_n)_{n\ge 1}$, where $\sS_n$ is a collection of subsets of $\R^n$ such that 
\begin{itemize}
    \item $\sS_n$ is a boolean algebra with respect to the usual set-theoretic operations,
    \item $\sS_n$ contains all semialgebraic subsets of $\R^n$,
    \item $\sS$ is stable by cartesian products and linear projections,
    \item each $S \in \sS_1$ has only  finitely many connected components.
\end{itemize}
A set $S$ that belongs to $\sS$ is said to be \emph{definable} (in $\sS$).
A map $f : S \to \R^m$, $S\subseteq \R^n$, is called \emph{definable} if its graph 
$\{(x,f(x)): x \in S\}$ is a definable subset of $\R^n \times \R^m$. 

The basic example of an o-minimal expansion of the real field is the family of semialgebraic sets. 
Another important example is the family of globally subanalytic sets. 
Many more interesting o-minimal structures have been identified in recent decades.
We refer to \cite{vandenDries98} and \cite{vandenDriesMiller96} for the fundamentals of the theory. 

From now on, the attribute ``definable'' will refer to a fixed o-minimal expansion of the real field.

\subsection{Spaces of differentiable functions} \label{ssec:spaces}
Let $\om$ be a modulus of continuity, 
i.e., a positive, continuous, increasing, and concave function $\om : (0,\infty) \to (0,\infty)$ 
such that $\om(t) \to 0$ as $t \to 0$.
Let $C^{0,\om}(\R^n)$ be the set of all continuous bounded functions $f : \R^n \to \R$ such that 
\[
    |f|_{C^{0,\om}(\R^n)} := \inf\{C>0 : |f(x)-f(y)| \le C\, \om(\|x-y\|) \text{ for all }x,y \in\R^n\} < \infty.
\]
For $m$ a nonnegative integer, $C^{m,\om}(\R^n)$ consists of all $C^m$-functions such that $\p^\al f$ is 
globally bounded on $\R^n$, for all $|\al|\le m$, and $\p^\al f \in C^{0,\om}(\R^n)$, for all $|\al|=m$.
Then $C^{m,\om}(\R^n)$ is a Banach space with the norm
\begin{equation} \label{eq:Cmomnorm}
    \|f\|_{C^{m,\om}(\R^n)} := \sup_{x \in \R^n} \sup_{|\al|\le m} |\p^\al f(x)| + \sup_{|\al|= m}|\p^\al f|_{C^{0,\om}(\R^n)}.
\end{equation}
Let $C^{m,\om}_{\on{def}}(\R^n)$ be the subspace consisting of the functions in  $C^{m,\om}(\R^n)$ that are definable.

Let $X \subseteq \R^n$ be a subset.
The trace space $C^{m,\om}(\R^n)|_X$ is the set of all $f : X \to \R$ such that there exists $F \in C^{m,\om}(\R^n)$ with $F|_X =f$.
It carries the norm
\[
    \|f\|_{C^{m,\om}(\R^n)|_X} := \inf \{\|F\|_{C^{m,\om}(\R^n)} : F \in C^{m,\om}(\R^n),\, F|_X = f\}.
\]
Similarly, we consider the space $C^{m,\om}_{\on{def}}(\R^n)|_X$ of restrictions to $X$ of functions in $C^{m,\om}_{\on{def}}(\R^n)$
with the norm
\[
    \|f\|_{C^{m,\om}_{\on{def}}(\R^n)|_X} := \inf \{\|F\|_{C^{m,\om}(\R^n)} : F \in C^{m,\om}_{\on{def}}(\R^n),\, F|_X = f\}.
\]
If $X$ is definable, then each element of $C^{m,\om}_{\on{def}}(\R^n)|_X$
is a definable function on $X$.

\subsection{Pseudometric spaces} \label{ssec:pseudometric}
By a \emph{pseudometric space} $(\cM,\rh)$ we mean a non-empty set $\cM$ together with a nonnegative 
real valued function $\rh : \cM \times \cM \to [0,\infty)$ 
such that $\rh(x,x) = 0$, $\rh(x,y) = \rh(y,x)$, and $\rh(x,y) \le \rh(x,z) + \rh(z,y)$ for all $x,y,z \in \cM$.
The function $\rh$ is called a \emph{pseudometric};
if additionally $\rh(x,y)=0$ implies $x=y$, it is called a \emph{metric}.
We say that $(\cM,\rh)$ is an \emph{extended pseudometric space} if the pseudometric $\rh$ may also take the value $+\infty$.   

A pseudometric space $(\cM,\rh)$ is called \emph{definable} if $\cM$ is a definable subset of $\R^N$, for some $N$,
and $\rh : \cM \times \cM \to [0,\infty)$ is a definable function.

A map $f : \cM \to \R^n$ is \emph{Lipschitz} if 
\[
    |f|_{\on{Lip}(\cM,\R^n)} := \inf \{C>0 : \|f(x) - f(y)\| \le C \, \rh(x,y) \text{ for all } x,y \in \cM\} <\infty.
\]
Let $\om$ be a modulus of continuity.
A map $f : X \to \R^n$ with $X \subseteq \R^N$ is \emph{$\om$-H\"older} if 
\[
    |f|_{C^{0,\om}(X,\R^n)} := |f|_{\on{Lip}(X_\om,\R^n)} < \infty,
\]
where $X_\om$ is the space $X$ endowed with the metric $\rh(x,y) := \om(\|x-y\|)$.

\subsection{Set-valued mappings} \label{ssec:setvalued}

A \emph{set-valued mapping} is a map $F : X \to 2^Y$. The \emph{graph} of $F$ is the subset $\bigcup_{x \in X} (\{x\} \times F(x))$ of $X \times Y$.
A \emph{selection} of $F$ is a map $f : X \to Y$ such that $f(x) \in F(x)$ for all $x \in X$. 

If $X\subseteq \R^N$ is definable and $Y=\R^n$, then we say that a set-valued map $F : X \to 2^{\R^n}$ is \emph{definable} if 
its graph is a definable subset of $\R^N \times \R^n$.

\section{Definable Lipschitz selections} \label{sec:selection}

The purpose of this section is to formulate and prove \Cref{thm:BS23def} 
from which \Cref{thm:B} and \Cref{thm:C} will follow easily.

\subsection{Definable Lipschitz selections for cube-valued maps}

Let $\cQ(\R^n)$ be the family of all closed cubes $Q(x,r) := \{y \in \R^n : \|x-y\| \le r\}$, where $r \in [0,\infty]$; 
thus $Q(x,0) = \{x\}$ and $Q(r,\infty)=\R^n$ count as cubes and belong to $\cQ(\R^n)$.  
Cubes centered at the origin will be denoted by $Q(r) := Q(0,r)$. 

\begin{lemma} \label{lem:BS26def}
    Let $(\cM,\rh)$ be a definable pseudometric space. 
    Let $F : \cM \to \cQ(\R^n)$ be a definable map. 
    Assume that for every $2$-point subset $\cN \subseteq \cM$ there exists a Lipschitz selection $f_\cN : \cN \to \R^n$ 
    of $F|_\cN$ with $|f_\cN|_{\on{Lip}(\cN,\R^n)}\le 1$.
    Then there exists a definable Lipschitz selection $f : \cM \to \R^n$ 
    of $F$ with $|f|_{\on{Lip}(\cM,\R^n)}\le 1$.
\end{lemma}

\begin{proof}
    Cf.\ \cite[Proposition 1.24 and Corollary 1.25]{BrudnyiBrudnyi12Vol1}.  
    By projection to the coordinate axes of $\R^n$ it is enough to consider the case $n = 1$.
    Then $F(m)$ is a closed interval $[a(m),b(m)]$ (either bounded, possibly a point, or $\R$).

    Let us first show that we may assume that all $F(m)$ are bounded.  
    Consider the definable set $\cM_0 := \{ m \in \cM : F(m) \text{ is bounded} \} = b^{-1}(\R)$
    and the definable map $F_0 := F|_{\cM_0}$.
    We may assume that $\cM_0$ is non-empty, since otherwise 
    $F(m) = \R$ for all $m \in \cM$ and $f:= 0$ is the required Lipschitz selection.
    If we prove that $F_0 : \cM_0 \to \cQ(\R)$ has a definable Lipschitz selection $f_0 : \cM_0 \to \R$ 
    with $|f_0|_{\on{Lip}(\cM_0,\R)} \le 1$, 
    then $f_0$ admits a definable Lipschitz extension $f : \cM \to \R$ with 
    $|f|_{\on{Lip}(\cM,\R)} \le 1$ defined by 
    \begin{equation} \label{eq:defLipext}
        f(m) := \inf_{m' \in \cM_0} (f_0(m') + \rh(m,m')), \quad m \in \cM,
    \end{equation}
    which clearly is the required selection of $F$ (since $F(m)=\R$ if $m \not\in \cM_0$).
    Let us check that $f$ in \eqref{eq:defLipext} has the desired properties. 
    It is finite, because 
    \[
        f_0(m_0) - \rh(m,m_0) \le f(m) \le f_0(m_0) + \rh(m,m_0)
    \]
    for any $m_0 \in \cM_0$, 
    where the first inequality holds, since $|f_0|_{\on{Lip}(\cM_0,\R)} \le 1$ entails
    \[
        f_0(x) + \rh(x,m) \ge f_0(m_0) - \rh(x,m_0) + \rh(x,m) \ge f_0(m_0) - \rh (m,m_0), 
    \]
    for all $x \in \cM_0$.
    We have $f|_{\cM_0} = f_0$, since 
    if $m \in \cM_0$ then 
    \[
        f(m) \le f_0(m) \le f_0(m') + \rh(m,m')  
    \]
    for all $m' \in \cM_0$, whence $f(m) = f_0(m)$.
    To see $|f|_{\on{Lip}(\cM,\R)} \le 1$ let $m_1, m_2 \in \cM$. Then
    \[
        f(m_1) \le  \inf_{m' \in \cM_0} (f_0(m') + \rh(m_2,m') + \rh(m_1,m_2)) = f(m_2) + \rh(m_1,m_2),
    \]
    and we are done.

    Hence, it suffices to prove the lemma under the assumption that $F(m)$ is bounded for all $m \in \cM$.
    We will show that 
    \[
        f(m) := \inf_{m' \in \cM} (b(m') + \rh(m,m')) = \inf_{m' \in \cM} (\sup F(m') + \rh(m,m')), \quad m \in \cM,
    \]
    is a definable Lipschitz selection of $F$ with $|f|_{\on{Lip(\cM,\R)}} \le 1$; in particular, $f$ is finite.
    It is immediate from the definition that $f$ is definable and satisfies $f(m) \le b(m)$ for every $m \in \cM$.
    To see $a(m) \le f(m)$ let $m' \in \cM$ be arbitrary. By assumption, there exists $g:=f_{\{m,m'\}} : \{m,m'\} \to \R$ 
    with $a(m) \le g(m) \le b(m)$, $a(m') \le g(m') \le b(m')$, and $|g(m)-g(m')| \le \rh(m,m')$, 
    so that  $a(m) \le b(m') + \rh(m,m')$. 

    Finally, for any $m_1,m_2 \in \cM$, 
    \[
        f(m_1) \le \inf_{m' \in \cM} (b(m') + \rh(m_2,m')+\rh(m_1,m_2)) = f(m_2) + \rh(m_1,m_2),
    \]
    and we conclude that $|f|_{\on{Lip}(\cM,\R)}\le 1$.
\end{proof}

\subsection{Definable Lipschitz selections for affine-set valued maps}

Let $V \subseteq \R^N$ be a definable set.
Let $\Ga = (V,E)$ be a graph with set of vertices $V$ 
and set of edges $E \subseteq \{(v,v') \in V \times V : v \ne v'\}$. 
We assume that $E$ is symmetric in the sense that $(v,v') \in E$ implies $(v',v) \in E$.
In other words, 
the graph is undirected and, strictly speaking, 
the set of edges actually is $E/\!\sim$, where $(v,v') \sim (v',v)$, but this 
inaccuracy in notation will cause no troubles.
Furthermore, we assume that $E$ is definable.
Let us write $v \leftrightarrow v'$ if $v, v' \in V$ are joined by an edge, i.e., 
$(v,v') \in E$. 

We say that a subset $W \subseteq V$ is \emph{admissible} if $(W,E_W)$ regarded as a subgraph of $\Ga=(V,E)$,
where $(v,v') \in E_W$ if and only if $(v,v') \in E$ and $v,v' \in W$,    
has no isolated vertices. This means that each vertex is joined by an edge to another vertex unless 
the graph consists of but one vertex.

Let us endow the graph $\Ga = (V,E)$ with a \emph{weight}, i.e., a symmetric function 
$w : E \to [0,\infty]$.  
This induces an extended pseudometric space $(V,\si)$, 
where $\si : V \times V \to [0,\infty]$ is defined
by
\begin{equation} \label{eq:definitionrho}
    \si(v,v') := \inf \sum_{i=0}^{k} w(e_i), \quad v \ne v',
\end{equation}
where the infimum is taken over all finite paths $\{e_i\}_{i=0}^k$, $k \in \N$, of edges in $\Ga$ joining $v$ and $v'$. 
Moreover,
$\si(v,v) := 0$ and $\si(v,v') :=+\infty$ if there is no path of edges joining $v\ne v'$. 

We say that $\Ga = (V,E,w)$ is a \emph{definable weighted graph} if the sets $V$ and $E$ are definable as specified above 
and, 
additionally, 
there exist a definable pseudometric $\rh : V \times V \to [0,\infty)$ on $V$ and an absolute constant $A \ge 1$ such that
\begin{equation} \label{eq:dwgraph}
    \frac{1}{A} \rh \le \si \le A \rh.
\end{equation}
So $(V,\si)$ is a pseudometric space, while $(V,\rh)$ is even a definable pseudometric space.  
In particular, in a definable weighted graph any two vertices are joined by a path of edges 
and $V$ is an admissible subset of $V$.

\begin{remark} \label{rem:dpm}
    Let $(\cM,\rh)$ be a definable pseudometric space.
    The full graph with set of vertices $\cM$, set of edges $\{(m,m') \in \cM \times \cM : m \ne m'\}$,
    and weight $\rh$ is a definable weighted graph (where $\si = \rh$).
    Furthermore, any subset of $\cM$ is admissible.
\end{remark}

Recall that $\cA_k(\R^n)$ is the family of affine subspaces of $\R^n$ of dimension at most $k$
and that a map $F : V \to \cA_k(\R^n)$ is called definable if its graph 
is a definable subset of $\R^N \times \R^n$. 
For a map $f : W \to \R^n$, $W \subseteq V$, we set 
\[
    |f|_{\on{Lip}(W,\R^n)} := \inf \{C>0 : \|f(v) - f(v')\| \le C \, \rh(v,v') \text{ for all } v,v' \in W\};
\]
i.e., the Lipschitz seminorm is computed in terms of $\rh$ (not $\si$).

\begin{theorem} \label{thm:BS23def}
    Let $\Ga = (V,E,w)$ be a definable weighted graph.
    Let $F : V \to \cA_k(\R^n)$ be a definable map.
    Assume that for each admissible subset $W \subseteq V$ of cardinality $\# W\le 2^{k+1}$ 
    there exists a Lipschitz selection $f_W : W \to \R^n$ of $F|_W$ 
    with $|f_W|_{\on{Lip}(W,\R^n)} \le 1$. 
    Then there exists a definable Lipschitz selection $f : V \to \R^n$ of $F$ with 
    $|f|_{\on{Lip}(V,\R^n)} \le C$, where $C\ge 1$ is a constant depending only on $k$, $n$, and 
    $A$ from \eqref{eq:dwgraph}.
\end{theorem}

Note that \Cref{thm:BS23def} implies \Cref{thm:B} in view of \Cref{rem:dpm}. 
It is a definable version of \cite[Theorem 2.3]{Brudnyi:2001aa}.  
We closely follow its proof making sure that all steps are definable.

Let us give a rough outline of the general strategy of the proof.
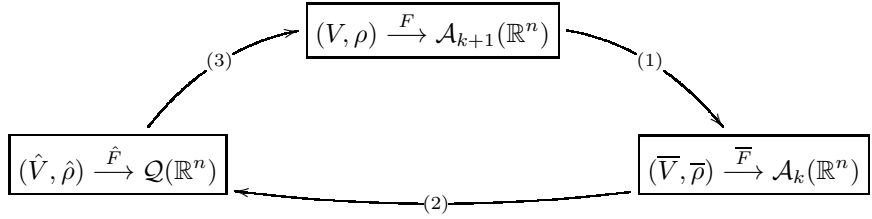
\[
    \xymatrix{
      & \boxed{ (V,\rh) \stackrel{F}{\longrightarrow} \cA_{k+1}(\R^n) } \ar@/^2pc/[dr]|(.6){(1)}   &
      \\
      \boxed{ (\hat V,\hat \rh) \stackrel{\hat F}{\longrightarrow} \cQ(\R^n) } \ar@/^2pc/[ur]|(.4){(3)}
      && 
      \boxed{ (\ol V,\ol \rh) \stackrel{\ol F}{\longrightarrow} \cA_{k}(\R^n) } \ar@/^1pc/[ll]|{(2)}  
  }
\]

It is by induction on $k$. 
The induction step ($k \to k+1$)
is divided into the following three stages (which correspond to the bent arrows in the diagram).    
\begin{enumerate}
    \item From $F : V \to \cA_{k+1}(\R^n)$ one constructs a ``doubling'' $(\ol V, \ol \rh)$ 
        of the space $(V,\rh)$ and a map $\ol F : \ol V \to \cA_k(\R^n)$ 
        in such a way that the induction hypothesis yields a Lipschitz selection $\ol f$ of $\ol F$.
    \item The Lipschitz selection $\ol f$ of $\ol F$ 
        is used to define a new space $(\hat V = \ol V, \hat \rh)$ and 
        a cube-valued map $\hat F : \hat V \to \cQ(\R^n)$; 
        the center of the cube $\hat F(\ol v)$ is $\ol f(\ol v)$. 
        Thanks to \Cref{lem:BS26def}, there is a Lipschitz selection $\hat f$ of $\hat F$ 
        which moreover can be interpreted as a Lipschitz map defined on $(V,\rh)$.
    \item The desired Lipschitz selection $f$ of $F$ is finally found 
        by defining $f(v)$ to be the orthogonal projection of $\hat f(v)$ 
        to the affine subspace $F(v)$ of $\R^n$.
\end{enumerate}

\subsection{Proof of \Cref{thm:BS23def}}

We prove \Cref{thm:BS23def} by induction on $k$.

The base case $k=0$ is trivial: $\cA_0(\R^n) \cong \R^n$ so that $f = F : V \to \R^n$ is the desired definable 
Lipschitz selection. 
Indeed, let $v,v' \in V$ be any two distinct vertices. 
Then there is a path of edges $v=:v_0 \leftrightarrow v_1 \leftrightarrow \cdots \leftrightarrow v_\ell := v'$ 
joining $v$ and $v'$ so that, by assumption (as $\{v_i,v_{i+1}\}$ is admissible) and \eqref{eq:dwgraph},
\begin{equation} \label{eq:2point}
    \|f(v) - f(v')\| \le \sum_{i=0}^{\ell-1} \rh(v_i,v_{i+1}) 
    \le A \sum_{i=0}^{\ell-1} w(v_i,v_{i+1}).
\end{equation}
Taking the infimum over all paths joining $v$ and $v'$ and using again \eqref{eq:dwgraph},
we conclude that $|f|_{\on{Lip}(V,\R^n)} \le A^2$.

Let us assume that the result holds for some $k<n$ and prove it for $k+1$. 
Let $F : V \to \cA_{k+1}(\R^n)$ be a definable map and assume that 
for each admissible subset $W \subseteq V$ of cardinality $\# W\le 2^{k+2}$ 
there exists a Lipschitz selection $f_W : W \to \R^n$ of $F|_W$ 
with $|f_W|_{\on{Lip}(W,\R^n)} \le 1$.

Let $v_1,v_2 \in V$ be such that $v_1 \leftrightarrow v_2$.
Then $\{v_1,v_2\}$ is an admissible subset of $V$ and 
the assumption  
implies that 
there exist points $x_1 \in F(v_1)$, $x_2 \in F(v_2)$ such that 
\begin{equation} \label{eq:x1x2}
    \|x_1 - x_2\| \le \rh(v_1,v_2).
\end{equation} 
It follows that the set
\[
    \{(x_1,x_2) \in F(v_1) \times F(v_2): \|x_1 - x_2\| \le  \rh(v_1,v_2)\}
\] 
is non-empty and definable, since $F$ and $\rh$ are definable.
Therefore, by definable choice, we may assume that $x_i=x_i(v_1,v_2)$, $i=1,2$, depend in a definable way on $v_1$ and $v_2$.
We may conclude that 
\begin{equation}
    \label{eq:C}    
    P(v_1,v_2) := F(v_1) \cap \big(F(v_2) + Q(2 \rh(v_1,v_2)) + x_1(v_1,v_2)-x_2(v_1,v_2)\big)
\end{equation}
is a non-empty convex definable set that is symmetric with respect to the point $x_1$.

Let us understand better the geometry of $P(v_1,v_2)$.
For ease of notation let $U_1 := F(v_1)$ and $U_2 := F(v_2) + x_1 - x_2$.
Then $U_1$ and $U_2$ are affine subspaces of $\R^n$ through $x_1$ of dimension at most $k+1$.
Set $d_i := \dim U_i$, $i=1,2$.
Moreover, the ``slab'' $S_2 := U_2 + Q(2\rh(v_1,v_2))$ is the thickening of $U_2$ by 
the cube $Q(2 \rh(v_1,v_2))$.

If $d_1 =0$, then 
\begin{align*}
    P(v_1,v_2) = U_1  \cap S_2  = \{x_1\} = U_1 \cap \big(\{x_1\} + Q(0) \big).
\end{align*}

Now assume that $d_1 >0$.
If $U_1 \cap U_2 = \{x_1\}$, 
then $P(v_1,v_2)$ is a polytope in $U_1$ centered at $x_1$ with pairwise parallel faces. 
In fact, the pairs of parallel faces result as the intersections with $U_1$ of  
opposite parallel faces of the slab $S_2$. 
Let $p$ be the number of pairs of parallel faces of $P(v_1,v_2)$.
Thus $p \le n$. 
(Note that $p=n$ may occur if $d_2 =0$ so that $S_2$ is 
a cube.)
Let $L_1,\ldots,L_p$
be the affine subspaces of $U_1$ of dimension $d_1 -1$ through $x_1$
that are equidistant to the respective affine hulls of opposite parallel faces of $P(v_1,v_2)$, 
the distance being computed with respect to $\|\cdot \|$ and
having the respective values $r_1,\ldots,r_p$.
Then 
\begin{equation} \label{eq:Liprelim} 
    P(v_1,v_2) = \bigcap_{i=1}^{p} \big( U_1 \cap (L_i + Q(r_i)) \big). 
\end{equation}

Now suppose that $U_1 \cap U_2 \ne \{x_1\}$. 
If $U_1 \cap U_2 \ne U_1$,
then $P(v_1,v_2)$ is a polytope in $U_1$ centered at $x_1$ 
with pairwise parallel faces 
and infinite extension in the directions of $U_1 \cap U_2$.
The number of pairs of parallel faces $p$ again satisfies $p \le n$. 
Let $L_1,\ldots,L_{p}$ be the affine subspaces of $U_1$ of dimension $d_1-1$ 
through $x_1$
that are equidistant 
to the respective affine hulls of 
\emph{finite} opposite parallel faces of $P(v_1,v_2)$. 
As before the distance is computed with respect to $\|\cdot \|$ and 
let $r_1,\ldots,r_{p}$ be the respective values. 
Then we again have the representation \eqref{eq:Liprelim}.

It remains to consider the case $U_1 \cap U_2 = U_1$. 
Then $P(v_1,v_2) = U_1$. 
If $d_1 \le k$, we set $L_1 := U_1$ and $r_1 :=0$ so that \eqref{eq:Liprelim}
remains valid with $p=1$.
If $d_1 = k+1$, we define $L_1 := \{x_1\}$ and $r_1 := \infty$. 
Again  \eqref{eq:Liprelim}
holds true with $p=1$. 
Note that this is the only case, where we allow an infinite radius. 
This case occurs if and only if $F(v_1)$ and $F(v_2)$ are parallel and both have dimension $k+1$.

Thus we proved that in any case 
\begin{equation}
    \label{eq:Li}
    P(v_1,v_2) = \bigcap_{i \in I(v_1,v_2)} \big( F(v_1) \cap (L_i + Q(r_i)) \big), 
\end{equation}
where the index set $I(v_1,v_2)$ has cardinality at least $1$ and at most $n$
and where each $L_i$ is an affine subspace of $F(v_1)$ through $x_1$ of dimension at most $k$. 
The radii $r_i$ are finite unless $F(v_1) \| F(v_2)$ and $\dim F(v_1) = \dim F(v_2) = k+1$.

Let us consider the definable sets 
\[
    \ol V_{\|} := \{(v_1,v_2) : v_1 \leftrightarrow v_2,\, F(v_1) \| F(v_2),\, \dim F(v_1) = \dim F(v_2) = k+1\}.
\]
and 
\[
    \ol V_0 := \{\ol v = (v_1,v_2,i) : v_1 \leftrightarrow v_2,\, (v_1,v_2) \not\in \ol V_{\|},\, i \in I(v_1,v_2)\}.
\]
We get a new definable pseudometric space
$(\ol V_0,\ol \rh)$,
where, for $\ol v = (v_1,v_2,i) \ne \ol v' = (v_1',v_2',i')$,  
\[
    \ol\rh(\ol v, \ol v') := \rh(v_1,v_1') + r_{i} + r_{i'}.
\]
We extend this data to an extended pseudometric space $(\ol V,\ol \rh)$ by 
defining 
\[
    \ol V := \{\ol v = (v_1,v_2,i) : v_1 \leftrightarrow v_2,\, i \in I(v_1,v_2)\}
\]
and putting $\ol \rh(\ol v, \ol v') := +\infty$ if at least one of $\ol v \ne \ol v'$ 
does not belong to $\ol V_0$.

Define $\ol F : \ol V \to  \cA_k(\R^n)$ by setting $\ol F(v_1, v_2, i) := L_i$. 
By the above discussion, $\ol F$ is definable.

\begin{claim} \label{claim1}
    There exists a definable Lipschitz selection $\ol f : \ol V \to  \R^n$ of $\ol F$ 
    with $|\ol f|_{\on{Lip}(\ol V, \R^n)} \le C(k,n)$.  
\end{claim}

We will first consider the restriction $\ol F_0 := \ol F|_{\ol V_0}$.  
By \Cref{rem:dpm}, we may view $(\ol V_0,\ol \rh)$ as a definable weighted graph, 
where each subset of vertices is admissible (and $A=1$).
Thus, by the induction hypothesis,
it suffices to show that for every subset $\ol W \subseteq \ol V_0$ of cardinality $\# \ol W \le 2^{k+1}$ 
there is a Lipschitz selection $\ol f_{\ol W} : \ol W \to \R^n$ of $\ol F|_{\ol W}$ with 
$|\ol f_{\ol W}|_{\on{Lip}(\ol W, \R^n)} \le 1$. 

For $\ol v = (v_1,v_2,i)$ we write 
\[
    \on{pr}_1(\ol v) = v_1(\ol v) = v_1,\quad \on{pr}_2(\ol v) =v_2(\ol v)= v_2, \quad \on{pr}_3(\ol v) = i(\ol v)= i. 
\]
Set $W := \on{pr}_1 \ol W \cup \on{pr}_2 \ol W$.
Then $W$ is an admissible subset of $V$ of cardinality $\# W \le 2 \ol W \le 2^{k+2}$, since for each $v \in W$ 
there is $v'\in W$ such that $v \leftrightarrow v'$. 
Thus, by the assumption on $F$, there exists a Lipschitz selection $f_W : W \to \R^n$ of $F|_W$ 
with $|f_W|_{\on{Lip}(W,\R^n)} \le 1$. 

Let $v_1,v_2 \in W$. Then $f_W(v_1) \in F(v_1)$ and, by \eqref{eq:x1x2}, 
\begin{align*}
    \MoveEqLeft
    \|f_W(v_1) - f_W(v_2) - x_1(v_1,v_2)+ x_2(v_1,v_2)\| 
    \\
  &\le \|f_W(v_1) - f_W(v_2)\| + \| x_1(v_1,v_2) -  x_2(v_1,v_2)\| 
  \\
  &\le  2\, \rh(v_1,v_2).
\end{align*}
Consequently, $f_W(v_1) \in P(v_1,v_2)$; cf.\ \eqref{eq:C}.

Now define $\ol f_{\ol W} : \ol W \to \R^n$ by 
letting $\ol f_{\ol W}(\ol v)$ be a point in $\ol F(\ol v) = L_{i(\ol v)}$ which minimizes the $\|\cdot \|$-distance to 
$f_{W}(v_1(\ol v))$. 
Since $f_{W}(v_1(\ol v)) \in P(v_1(\ol v),v_2(\ol v))$, 
\eqref{eq:Li} implies that 
$f_{W}(v_1(\ol v)) \in L_{i(\ol v)} + Q(r_{i(\ol v)})$. 
Then (by the minimality property of $\ol f_{\ol W} (\ol v)$) we have
\[
    \|\ol f_{\ol W} (\ol v) - f_{W}(v_1(\ol v))\| \le r_{i(\ol v)},
\] 
and consequently
\begin{align*}
    \|\ol f_{\ol W} (\ol v) - \ol f_{\ol W}(\ol v')\| 
  &\le \|f_{W} (v_1(\ol v)) - f_{W}(v_1(\ol v'))\| 
  +  r_{i(\ol v)} + r_{i(\ol v')}  
  \\
  &\le \rh(v_1(\ol v)),v_1(\ol v')) 
  +  r_{i(\ol v)} + r_{i(\ol v')}  
  \\
  &=  \ol \rh(\ol v,\ol v').
\end{align*}
Thus, $|\ol f_{\ol W}|_{\on{Lip}(\ol W,\R^n)} \le 1$.

By the induction hypothesis,
we may infer that there is a definable 
Lipschitz selection $\ol f_0 : \ol V_0 \to  \R^n$ of $\ol F_0$ 
with $|\ol f_0|_{\on{Lip}(\ol V_0, \R^n)} \le C(k,n)$.  

Let us extend $\ol f_0$ to a map $\ol f : \ol V \to \R^n$ by 
setting 
\[
    \ol f(\ol v) := 
    \begin{cases}
        \ol f_0(\ol v) & \text{ if } \ol v \in \ol V_0,
        \\
        x_1(\ol v) & \text{ if } \ol v \not\in \ol V_0.   
    \end{cases}
\]
Then $\ol f$ is obviously a definable selection of $\ol F$. 
Evidently, $|\ol f|_{\on{Lip}(\ol V, \R^n)} \le C(k,n)$, 
since $\ol \rh (\ol v,\ol v') = \infty$ whenever at least one of $\ol v \ne \ol v'$ 
does not belong to $\ol V_0$. 

Thus \Cref{claim1} is proved.

\begin{claim} \label{claim2}
    There is a definable Lipschitz map $\hat f : \ol V \to \R^n$ such that  
    \begin{enumerate}
        \item[(i)] $\hat f (\ol v)$ depends only on $v_1 (\ol v)$; thus we may regard $\hat f$ as a map defined on $V$, 
        \item[(ii)] $|\hat f|_{\on{Lip}(V,\R^n)} \le C$,
        \item[(iii)] $\| \hat f(\ol v) - \ol f(\ol v)\|\le C\, r_{i(\ol v)}$ for all $\ol v \in \ol V$,
    \end{enumerate} 
    where $C=C(k,n)$ is the constant from \Cref{claim1}.
\end{claim}

In order to define $\hat f$, we consider the definable pseudometric space $(\hat V,\hat \rh)$, where $\hat V := \ol V$ and 
\[
    \hat \rh(\ol v,\ol v') := C \, \rh(v_1 (\ol v), v_1 (\ol v')).
\]
Let $\hat F : \hat V \to \cQ(\R^n)$ be defined by $\hat F(\ol v) := Q(\ol f(\ol v), C r_{i(\ol v)})$. 
Then $\hat F$ is definable. 
We will show that $\hat F$ has a definable Lipschitz selection $\hat f : \hat V \to \R^n$ such that 
$|\hat f|_{\on{Lip}(\hat V,\R^n)} \le 1$. Then \Cref{claim2} follows easily:
we have
\[
    \|\hat f(\ol v) - \hat f(\ol v')\| \le \hat \rh(\ol v,\ol v') = C \, \rh(v_1 (\ol v), v_1 (\ol v'))
\]
so that $\hat f(\ol v) = \hat f(\ol v')$ if $v_1 (\ol v)= v_1 (\ol v')$ which implies (i) and (ii).
Since $\hat f$ is a selection of $\hat F$, (iii) is clear.

Let $\ol v,\ol v' \in \hat V = \ol V$ be any two distinct points.
By \Cref{claim1},
\[
    \|\ol f(\ol v) - \ol f(\ol v')\| \le C \, \ol \rh(\ol v, \ol v') = C\, r_{i(\ol v)} + C\, r_{i(\ol v')} + C\,\rh(v_1(\ol v),v_1(\ol v')).
\]
This means that the cubes $\hat F(\ol v)$ and $\hat F(\ol v')$ have $\|\cdot\|$-distance at most $C\,\rh(v_1(\ol v),v_1(\ol v'))$.
So there exist points $y \in \hat F(\ol v)$ and $y' \in \hat F(\ol v')$ 
such that 
\[
    \|y-y'\|\le C\,\rh(v_1(\ol v),v_1(\ol v')) = \hat \rh(\ol v,\ol v').
\]
In other words, for every $2$-point subset $\{\ol v,\ol v'\}$ of $\hat V$ the restriction $\hat F|_{\{\ol v,\ol v'\}}$ has 
a Lipschitz selection $\hat f_{\{\ol v,\ol v'\}}$ with $|\hat f_{\{\ol v,\ol v'\}}|_{\on{Lip}(\{\ol v,\ol v'\},\R^n)} \le 1$.
By \Cref{lem:BS26def}, there is a definable Lipschitz selection $\hat f : \hat V \to \R^n$ of $\hat F$ with 
$|\hat f|_{\on{Lip}(\hat V,\R^n)} \le 1$. This ends the proof of \Cref{claim2}.

\begin{claim} \label{claim3}
    The desired definable Lipschitz selection $f : V \to \R^n$ of $F$ with $|f|_{\on{Lip}(V,\R^n)} \le C_1(k,n,A)$ is given by
    \[
        f(v) := \on{pr}_{F(v)} \hat f(v),
    \]
    i.e., $f(v)$ is the orthogonal projection of $\hat f(v)$ to the affine subspace $F(v)$ of $\R^n$. 
\end{claim}

It is clear that $f$ is a definable selection of $F$. To check the Lipschitz property 
we show that 
\begin{equation} \label{eq:fLip}
    \|f(v) - f(v')\| \le C_1 \, \rh(v,v'), \quad \text{ whenever }  v \leftrightarrow v', 
\end{equation}
where $C_1 = C_1(k,n)$. 
This suffices 
in view of \eqref{eq:definitionrho} and \eqref{eq:dwgraph} (and a computation similar to \eqref{eq:2point}
after which $C_1$ will also depend on $A$).

In order to prove \eqref{eq:fLip}, we need some technical facts.
First of all, we may assume that $\dim F(v) \ge \dim F(v')$; 
if not we just interchange the roles of $v$ and $v'$.

\begin{claim} \label{claim4}
    The cube $K_v := \bigcap \{Q(f(v),r_{i(\ol v)}) : \ol v \in \ol V,\, \on{pr}_1 (\ol v) = v \}$ satisfies 
    \[ 
        \hat f(v) \in \la K_v,
    \] 
    where $\la := (1+\sqrt n) C$ and $C$ is the constant from \Cref{claim1},
    and
    \[
        K_v \cap F(v) \subseteq F(v') + Q(2 \la \rh(v,v')) + x_1(v,v') - x_2(v,v')
    \]
    whenever $v \leftrightarrow v'$. 
    (Recall that $\la K_v$ denotes the cube with the same center as $K_v$ and 
    $\la$ times its radius.)
\end{claim} 

Since $\ol v := (v,v',i) \in \ol V$, for $i \in I(v,v')$, we have, by \Cref{claim2}(iii), 
\begin{equation*}
    \|\hat f(v) - \ol f(\ol v)\| \le C\, r_{i(\ol v)}.
\end{equation*}
As $\ol f(\ol v) \in \ol F(\ol v) = L_{i(\ol v)} \subseteq F(v)$, the definition of $f(v)$ implies 
\begin{align*}
    \|f(v) - \ol f(\ol v)\| &\le \|f(v) - \ol f(\ol v)\|_2 
    \\
                            &\le \|\hat f(v) - \ol f(\ol v)\|_2
    \le
    \sqrt n\, \|\hat f(v) - \ol f(\ol v)\| \le \sqrt n\, C\, r_{i(\ol v)}.   
\end{align*} 
These inequalities easily yield $\hat f(v) \in \la K_v$. 

If $x \in K_v$, then 
\[
    \|x - f(v)\|\le r_{i(\ol v)}, \quad \text{ for all } \ol v \in \ol V \text{ such that } \on{pr}_1 (\ol v) = v,
\]
whence 
\[
    \|x - \ol f(\ol v)\|\le \la r_{i(\ol v)}.
\]
Thus, and by the fact that  $I(v,v') = \{i(\ol v) :  \ol v \in \ol V,\, \on{pr}_1 (\ol v) = v, \, \on{pr}_2 (\ol v) = v'\}$, 
\[
    K_v \subseteq \bigcap_{i \in I(v,v')} Q(\ol f(\ol v), \la r_{i}) \subseteq \bigcap_{i \in I(v,v')} \big( L_i + Q(\la r_{i}) \big).
\]
Intersecting both sides with $F(v)$ and using \eqref{eq:C} and \eqref{eq:Li} (modulo a dilation with center $x_1(v,v')$ by the factor $\la$), 
also the second assertion in \Cref{claim4} follows.

\medskip

Let us continue with the proof of \Cref{claim3}.
By \Cref{claim4}, 
there exists $y(v,v') \in F(v') +  x_1(v,v') - x_2(v,v')$ with
\begin{equation} \label{eq:y}
    \|f(v) - y(v,v')\| \le 2 \la\, \rh(v,v'). 
\end{equation}
We consider the translate $F(v,v')$ of $F(v')$ through $f(v)$ defined by 
\[
    F(v,v') := F(v') + x_1(v,v') - x_2(v,v') - y(v,v') + f(v).
\]
Let us prove \eqref{eq:fLip}. For $v \leftrightarrow v'$,
we have   
\begin{align*}
    \|f(v) - f(v')\| &\le \|\on{pr}_{F(v)} \hat f(v) - \on{pr}_{F(v,v')} \hat f(v)\| 
                   \\&\quad + 
                   \|\on{pr}_{F(v,v')} \hat f(v) - \on{pr}_{F(v,v')} \hat f(v')\| 
                   \\&\quad + \|\on{pr}_{F(v,v')} \hat f(v') - \on{pr}_{F(v')} \hat f(v')\|. 
\end{align*}
The second term on the right-hand side is bounded by 
\[
    \|\hat f(v) - \hat f(v')\|_2 \le \sqrt n\, \|\hat f(v) - \hat f(v')\| \le C \sqrt n\, \rh(v,v'),
\]
thanks to \Cref{claim2}. The third term is bounded by 
\[
    \|x_1(v,v') - x_2(v,v') - y(v,v') + f(v)\|_2 \le \sqrt n (1+ 2 \la)\,\rh(v,v'),
\]
in view of \eqref{eq:x1x2} and \eqref{eq:y}.

It remains to estimate the first term.
By \Cref{claim4} and \eqref{eq:y}, 
\[
    K_v \cap F(v) \subseteq F(v,v') + Q(4\la\, \rh(v,v')),
\]
and, by a dilation with center $f(v)$ and factor $\sqrt n\, \la$, we find
\[
    \big(\sqrt n\, \la  K_v\big) \cap F(v) \subseteq F(v,v') + Q(4\sqrt n\, \la^2 \rh(v,v')).
\]
Let $B$ be the biggest Euclidean ball with center $f(v)$ that is contained in $\sqrt n\, \la  K_v$. 
Then 
\begin{equation*}
    \la K_v \subseteq B \subseteq \sqrt n\, \la K_v.
\end{equation*}
Let $B(r)$ be the Euclidean ball with center $0$ and radius  
$r:= 4 n \la^2 \rh(v,v')$.
Then 
\[
    B \cap F(v) \subseteq F(v,v') + B(r), 
\]
and since the orthogonal projection of a point in $B \cap F(v)$ on $F(v,v')$ is contained in $B \cap F(v,v')$, 
we have 
\[
    B \cap F(v) \subseteq B \cap F(v,v') + B(r).
\]
Recall that we assumed that $\dim F(v) \ge \dim F(v') = \dim F(v,v')$ (see 
the paragraph before \Cref{claim4}).
Since both these affine spaces pass through the center $f(v)$ of $B$, we thus also have
\[
    B \cap F(v,v') \subseteq B \cap F(v) + B(r). 
\]
That means that for the Hausdorff distance of $B \cap F(v)$ and $B \cap F(v,v')$ we have
\[
    d_\cH(B \cap F(v),B \cap F(v,v')) \le r = 4 n \la^2 \rh(v,v').
\] 
Now it suffices to invoke \cite[Lemma 2.9]{Brudnyi:2001aa} (note that is uses $\hat f(v) \in \la K_v$ shown in \Cref{claim4}) 
which 
gives that 
\[
    \|\on{pr}_{F(v)} \hat f(v) - \on{pr}_{F(v,v')} \hat f(v)\| \le  d_\cH(B \cap F(v),B \cap F(v,v')).
\] 
This completes the proof of \Cref{claim3} and thus the proof of \Cref{thm:BS23def}.

\subsection{Proof of \Cref{thm:C}}
\label{sec:proofC}

Let us define 
\[
    F(x) := \Big\{(f_1,\ldots,f_M) \in \R^M : \sum_{j=1}^M A_{ij}(x) f_j = b_i(x),\, i = 1,\ldots, N \Big\}, \quad x \in X.  
\]
We assume that the system \eqref{eq:system} has a solution that is $\om$-H\"older. In particular, $F(x)$ is a non-empty 
affine subspace of $\R^M$ for each $x \in X$. 
Since the functions $A_{ij}$ and $b_i$ are assumed definable on $X$, we see that 
$F : X \to \cA_M(\R^M)$ is a definable map. 
Now \Cref{thm:B} (applied to $\cM = X$ with the metric $\rh(x,y) = \om(\|x -y\|)$) 
shows that the system \eqref{eq:system} has a definable $\om$-H\"older solution.
Also the statement \eqref{eq:thmC} on the $\om$-H\"older seminorms is immediate from \Cref{thm:B}.

\section{Definable $C^{1,\om}$ extension of functions} \label{sec:extension}

We will now work towards the proof of \Cref{thm:A}. 
First, we will establish a correspondence between definable Whitney jets of class $C^{1,\om}$ 
and definable Lipschitz selections of a special affine-set valued map.
It will allow us to bring to bear \Cref{thm:BS23def}.

\subsection{Definable Whitney jets vs.\ definable Lipschitz selections}

Let $\om$ be a modulus of continuity.  
We will henceforth assume 
that
\begin{equation} \label{eq:omless1}
    \om \le 1.
\end{equation}
This is no restriction regarding \Cref{thm:A}, since replacing $\om$ by $\ol \om := \min\{1,\om\}$ 
in \eqref{eq:Cmomnorm} gives equivalent norms:
\[
    \|f\|_{C^{m,\om}(\R^n)} \le \|f\|_{C^{m,\ol \om}(\R^n)}\le 3\, \|f\|_{C^{m,\om}(\R^n)}.
\]
The first inequality is immediate from $\ol \om \le \om$.
For the second inequality, it is enough to consider $m=0$ and 
show that $|f|_{C^{0,\ol \om}(\R^n)} \le 2 \, \|f\|_{C^{0,\om}(\R^n)}$. 
To this end let $t_0 := \inf \{t>0 : \om(t)\ge 1 \}$. If $\|x-y\| < t_0$, then 
\[
    |f(x)-f(y)| \le |f|_{C^{0,\om}(\R^n)}\, \om(\|x-y\|) = |f|_{C^{0,\om}(\R^n)}\, \ol \om(\|x-y\|),
\]
and, if $\|x-y\|\ge t_0$, 
\[
    \frac{|f(x)-f(y)|}{\ol \om(\|x-y\|)}= |f(x)-f(y)| \le 2 \sup_{x \in \R^n} |f(x)|,
\]
which implies the assertion.

Let $X$ be a closed definable subset of $\R^n$. 
Set 
\[
    \cM_X := \{(x,y) \in X \times X : x \ne y\} 
\]
and endow it with the metric $\rh_\om$ defined by 
\[
    \rh_\om((x,y),(x',y')) := \om(\|x-y\|) + \om(\|x'-y'\|) + \om(\|x-x'\|), 
\]
if $(x,y)\ne (x',y')$, and $0$ otherwise. (At this stage, we do not assume that $\om$ is definable, 
but eventually we will.)

Let $f : X \to \R$ be a definable function.
Consider the definable map $L_f : \cM_X \to \cA_{n-1}(\R^n)$ defined by 
\[
    L_f(x,y):= \{z \in \R^n : \langle z , x-y \rangle  = f(x)-f(y)\}. 
\]
Let us show that $f$ can be completed to a definable Whitney jet of class 
$C^{1,\om}$ on $X$ if and only if $L_f$ has a definable Lipschitz selection. 

Recall that $(f,g)$ is a \emph{definable Whitney jet of class $C^{1,\om}$ on $X$} if 
$f : X \to \R$ and $g : X \to \R^n$ are definable maps and
\begin{align*}
    \|(f,g)\|_{X,1,\om} &:= \sup_{x \in X} |f(x)| + \sup_{x \in X} \|g(x)\| + |(f,g)|_{X,1,\om} < \infty,
    \intertext{where}
    |(f,g)|_{X,1,\om} &:=  \sup_{\substack{x, y \in X \\ x \ne y}} \frac{|f(x)-f(y)- \langle g(y), x -y \rangle |}{\|x-y\| \, \om(\|x-y\|)}
    +  \sup_{\substack{x, y \in X \\ x \ne y}} \frac{\|g(x) - g(y)\|}{\om(\|x-y\|)}. 
\end{align*}

\begin{proposition} \label{prop:BS31def}
    Let $X \subseteq \R^n$ be a closed definable set and let $f : X \to \R$ be a bounded definable function. 
    Then the following conditions are equivalent:
    \begin{enumerate}
        \item There exists a bounded definable function $g : X \to \R^n$ such that 
            $(f,g)$ is a definable Whitney jet of class $C^{1,\om}$ on $X$.
        \item There exists a bounded definable Lipschitz selection $\ell : \cM_X \to \R^n$ of $L_f$. 
    \end{enumerate}
    If the equivalent conditions \thetag{1} and \thetag{2} hold, then
    \begin{equation} \label{eq:equivnorms2}
        \inf_g \|(f,g)\|_{X,1,\om}  \approx \sup_{x \in X} |f(x)| + \inf_\ell \Big\{ \sup_{(x,y) \in \cM_X} \|\ell(x,y)\| 
        + |\ell|_{\on{Lip}(\cM_X,\R^n)} \Big\},
    \end{equation} 
    where the infimum on the left-hand side is taken over all $g$ satisfying \thetag{1} and 
    the infimum on the right-hand side over all $\ell$ satisfying \thetag{2}.
\end{proposition}

Note that ``$\approx$'' in \eqref{eq:equivnorms2} means that any of the two sides is 
bounded by the other side up to a multiplicative factor which depends only on $n$.

\begin{proof}
    (1) $\Rightarrow$ (2)
    Let $g$ be as in (1) so that $(f,g)$ is a definable Whitney jet of class $C^{1,\om}$ on $X$.
    For $(x,y) \in \cM_X$ let $\ell(x,y)$ be the point in $L_f(x,y)$ closest to $g(x)$ (with respect to the Euclidean metric), i.e.,
    \begin{equation}
        \|g(x) - \ell(x,y)\|_2  = \on{dist}_2(g(x),L_f(x,y)).   
    \end{equation} 
    Then $\ell : \cM_X \to \R^n$ is a definable selection of $L_f$. Let us check that it is bounded and Lipschitz. 
    We have 
    \begin{align*}
        \|g(x) - \ell(x,y)\|_2 = \frac{|f(x) - f(y) - \langle g(x), x-y  \rangle |}{\|x-y\|_2}
    \end{align*}
    which is bounded by 
    \begin{align*}
        \|(f,g)\|_{X,1,\om}  \frac{\|x-y\|\, \om(\|x-y\|)}{\|x-y\|_2} \le  \|(f,g)\|_{X,1,\om},
    \end{align*}
    in view of $\om \le 1$ (cf.\ \eqref{eq:omless1}). 
    Thus 
    \[
        \|\ell(x,y)\| \le \|g(x)\|+ \|g(x) - \ell(x,y)\| \le 2\, \|(f,g)\|_{X,1,\om},
    \]
    i.e., $\ell$ is bounded with $\sup_{(x,y) \in \cM_X} \|\ell(x,y)\| \le 2\, \|(f,g)\|_{X,1,\om}$.

    For $(x,y), (x',y') \in \cM_X$, we have
    \begin{align*}
        \|\ell(x,y) - \ell(x',y')\| 
        &\le \|\ell(x,y) - g(x)\| + \|g(x)-g(x')\| +  \|g(x') - \ell(x',y')\|
        \\
        &\le \|(f,g)\|_{X,1,\om} \big(\om(\|x-y\|) + \om(\|x-x'\|)+ \om(\|x'-y'\|) \big)
        \\
        &=  \|(f,g)\|_{X,1,\om}\, \rh_\om((x,y), (x',y')),  
    \end{align*}
    consequently, $|\ell|_{\on{Lip}(\cM_X,\R^n)} \le \|(f,g)\|_{X,1,\om}$.

    (2) $\Rightarrow$ (1)
    Assume that $L_f$ has a bounded definable Lipschitz selection 
    $\ell : \cM_X \to \R^n$. 
    For ease of notation, set
    \[
        C_\ell:= \sup_{(x,y) \in \cM_X} \|\ell(x,y)\| 
        + |\ell|_{\on{Lip}(\cM_X,\R^n)}.
    \]
    If $x \in X$ is an isolated point of $X$, let $\hat x$ be a closest point in $X \setminus \{x\}$ (with respect to $\|\cdot\|$). 
    Otherwise, let $\hat x := x$. The point $\hat x$ can be assigned to $x$ in a definable way. 
    Define $g : X \to \R^n$ by 
    \[
        g(x) :=
        \begin{cases}
            \ell(x,\hat x) & \text{ if } x \text{ is an isolated point of } X,\\
            \lim_{X \ni y \to x} \ell(x,y) & \text{ otherwise. } 
        \end{cases}
    \]
    Then $g$ is definable and $\|g(x)\|\le C_\ell$ for all $x \in X$. The limit exists, since 
    \[
        \|\ell(x,y) - \ell(x,y')\|\le C_\ell\, \rh_\om((x,y),(x,y')) =C_\ell \big( \om(\|x-y\|) + \om(\|x-y'\|) \big).
    \]

    Let us check that $(f,g)$ is a Whitney jet of class $C^{1,\om}$ on $X$. 
    For each $x \in X$ let $(x_i)$ be a sequence in $X\setminus \{x\}$ (possibly stationary) such that $x_i \to \hat x$ and $\ell(x,x_i) \to g(x)$. 
    Then 
    \begin{align*}
        \|g(x) - g(x')\| &= \lim_{i\to \infty} \|\ell(x,x_i) - \ell(x',x'_i)\| 
        \\
                         &\le C_\ell \lim_{i\to \infty} \rh_\om((x,x_i),(x',x'_i))      
                         \\
                         &\le C_\ell \big( \om(\|x-\hat x\|) + \om(\|x'-\hat x'\|) + \om(\|x-x'\|) \big) \le 3C_\ell\, \om(\|x-x'\|).
    \end{align*}    
    Since $\ell(x,y) \in L_f(x,y)$, we have  
    \begin{align*}
        |f(x) - f(y) - \langle g(x) , x-y \rangle |
    &= 
    |\langle \ell(x,y) - g(x) , x-y \rangle |
    \\
    &= 
    \lim_{i \to \infty}
    |\langle \ell(x,y) - \ell(x,x_i) , x-y \rangle |
    \\
    &\le
    n\, \|x-y\| \limsup_{i \to \infty} \|\ell(x,y) - \ell(x,x_i)\| 
    \\
    &\le
    n\,C_\ell\, \|x-y\| \limsup_{i \to \infty} \rh_\om((x,y),(x,x_i))
    \\
    &\le n\,C_\ell\, \|x-y\|  \big( \om(\|x-y\|) + \om(\|x-\hat x\|) \big)
    \\
    &\le 2n\,C_\ell\, \|x-y\|\, \om(\|x-y\|).
    \end{align*}
    Thus $|(f,g)|_{X,1,\om} \le 2n\, C_\ell$.
\end{proof}

By a little trick, the boundedness condition for $\ell$ in \Cref{prop:BS31def} 
may be absorbed by the Lipschitz condition. 
To this end,
we add a point $*$ to $\cM_X$. 
We define $\widetilde \cM_X := \cM_X \cup \{*\}$ and extend the metric to $\widetilde \cM_X$  by setting 
$\widetilde \rh_\om|_{\cM_X \times \cM_X} := \rh_\om$ and $\widetilde \rh_\om ((x,y),*) = \widetilde \rh_\om (*,(x,y)) := 2$
for all $(x,y) \in \cM_X$ 
as well as $\widetilde \rh_\om (*,*) := 0$. 
(Since $\om\le 1$, the triangle inequality holds: for $m,m' \in \cM_X$ we have $\rh_\om(m,m') \le 3 \le 2 + 2 = \rh_\om(m,*) + \rh_\om(*,m')$.)

The map $L_f$ is extended to $\widetilde L_f : \widetilde \cM_X \to \cA_{n-1}(\R^n)$ by setting $\widetilde L_f|_{\cM_X} = L_f$ and 
$\widetilde L_f(*):= \{0\}$. 

\begin{proposition} \label{prop:BS32def} 
    Under the assumptions of \Cref{prop:BS31def}, items \thetag{1} and \thetag{2} in that proposition are further equivalent to 
    \begin{enumerate}
        \item[(3)]   There exists a definable Lipschitz selection $\widetilde \ell : \widetilde \cM_X \to \R^n$ of $\widetilde L_f$. 
    \end{enumerate}
    If the equivalent conditions \thetag{1}, \thetag{2}, and \thetag{3} hold, then we also have 
    \begin{equation}
        \inf_g \| (f,g) \|_{X,1,\om} \approx  \sup_{x \in X} |f(x)| + \inf_{\widetilde \ell}  |\widetilde \ell|_{\on{Lip}(\widetilde \cM_X,\R^n)},
    \end{equation}
    where the infimum on the left-hand side is taken over all $g$ satisfying \thetag{1} and 
    the infimum on the right-hand side over all $\widetilde \ell$ satisfying \thetag{3}.
\end{proposition}

\begin{proof}
    To see that (3) implies (2),
    note that 
    \[
        \frac{\|\widetilde \ell(x,y)\|}{2} = \frac{\|\widetilde \ell(x,y) - \widetilde \ell(*)\|}{\widetilde \rh_\om((x,y),*)} \le |\widetilde \ell|_{\on{Lip}(\widetilde \cM_X,\R^n)}
    \]
    for all $(x,y) \in \cM_X$, if $\widetilde \ell$ is a Lipschitz selection of $\widetilde L_f$.  

    Conversely, if $\ell$ is a bounded definable Lipschitz selection of $L_f$, then 
    the unique extension $\widetilde \ell|_{\cM_X} := \ell$ and $\widetilde \ell(*) := 0$ is a definable Lipschitz selection of $\widetilde L_f$.
    Indeed, for $(x,y) \in \cM_X$,
    \[
        \| \widetilde \ell(x,y) - \widetilde \ell(*) \| = \| \ell(x,y) \| \le C = \frac{C}{2}\, \widetilde \rh_\om ((x,y),*), 
    \]
    where $C = \sup_{(x,y) \in \cM_X} \|\ell(x,y)\|$.
\end{proof}

Next we endow $\widetilde \cM_X$ with a weighted graph structure. Let 
$\widetilde \cM_X$ be the set of vertices of this graph. 
Two vertices $(x,y), (x',y') \in \cM_X$ are joined by an edge if $\{x,y\} \cap \{x',y'\} \ne \emptyset$, 
and $*$ is joined by an edge to every $(x,y) \in \cM_X$. 
If $m,m' \in \widetilde \cM_X$ are joined by an edge, then we assign the weight 
\[
    w(m,m') := 
    \begin{cases}
        \om(\|x-y\|) + \om(\|x'-y'\|) & \text{ if } m = (x,y),\, m' = (x',y') \in \cM_X,
        \\
        2 & \text{ otherwise. }
    \end{cases}
\] 
Let $\si$ be the associated (extended) pseudometric; cf.\ \eqref{eq:definitionrho}.
By \cite[Proposition 3.3]{Brudnyi:2001aa} 
(and thanks to our general assumption $\om \le 1$; cf.\ \eqref{eq:omless1}), we have 
\begin{equation} \label{eq:Ais2}
    \frac{1}{2} \widetilde \rh_\om \le \si \le 2 \widetilde \rh_\om.   
\end{equation}

\begin{corollary} \label{cor:MXdwgraph}
    If $\om$ is definable, then $(\widetilde M_X,\widetilde \rh_\om)$ is a definable metric space
    and, equipped with the above weighted graph structure, it is a definable weighted graph.
\end{corollary}

\begin{proof}
    This follows from the definitions and \eqref{eq:Ais2}.  
\end{proof}

We recall a consequence of \cite[Proposition 3.2 and Corollary 3.4]{Brudnyi:2001aa} 
(which is an analogue of \Cref{prop:BS32def}, where sets and maps are not 
necessarily definable).

\begin{proposition}[{\cite[Proposition 3.5]{Brudnyi:2001aa}}] \label{prop:BS35def} 
    Let $m \ge 1$ be an integer.
    Let $X \subseteq \R^n$ and $f : X \to \R$ be given (not necessarily definable).
    Assume that
    the restriction $f|_Y$ to every subset $Y \subseteq X$ of cardinality $\# Y\le m$ has an 
    extension $F_Y\in C^{1,\om}(\R^n)$ with $\|F_Y\|_{C^{1,\om}(\R^n)} \le 1$.
    If $\cN$ is an admissible subset of $\widetilde \cM_X$ 
    of cardinality $\# \cN \le \frac{2}{3} m$, 
    then $\widetilde L_f|_\cN$ has a Lipschitz selection $\widetilde \ell_{\cN} : \cN \to \R^n$
    with $|\widetilde \ell_{\cN}|_{\on{Lip}(\cN,\R^n)} \le C(n)$.
\end{proposition}

We are ready for the goal of this section: 

\begin{theorem} \label{thm:BS13defjet}
    Let $\om$ be a definable modulus of continuity.  
    Let $X \subseteq \R^n$ be a closed definable set and $f : X \to \R$ a definable function.
    Assume that
    the restriction $f|_Y$ to every subset $Y \subseteq X$ of cardinality $\# Y\le 3 \cdot 2^{n-1}$ has an 
    extension $F_Y\in C^{1,\om}(\R^n)$ with $\|F_Y\|_{C^{1,\om}(\R^n)} \le 1$.
    Then there exists a bounded definable function $g : X \to \R^n$ such that $(f,g)$ is a definable Whitney jet of class $C^{1,\om}$ on $X$
    with
    $\|(f,g)\|_{X,1,\om} \le C(n)$. 
\end{theorem}

\begin{proof}
    Applying the assumption to $1$-point sets $Y \subseteq X$, we may conclude that $\sup_{x \in X} |f(x)| \le 1$. 
    The assumption and \Cref{prop:BS35def} give that for every admissible subset $\cN \subseteq \widetilde \cM_X$ of cardinality 
    $\# \cN \le \frac{2}{3} \cdot 3 \cdot 2^{n-1} = 2^n$ 
    there exists a Lipschitz selection $\widetilde \ell_{\cN} : \cN \to \R^n$ of $\widetilde L_f|_\cN$ 
    with $|\widetilde \ell_{\cN}|_{\on{Lip}(\cN,\R^n)} \le C(n)$. 
    By \Cref{thm:BS23def} and \Cref{cor:MXdwgraph},  
    there exists a definable Lipschitz selection $\widetilde \ell : \widetilde \cM_X \to \R^n$ of $\widetilde L_f$ 
    with $|\widetilde \ell|_{\on{Lip}(\widetilde \cM_X,\R^n)} \le C_1(n)$; note that $k=n-1$ and $A = 2$ in this case. 
    Now the assertion follows from \Cref{prop:BS32def}.
\end{proof}

\subsection{Proof of \Cref{thm:A}}

We shall see that
the identity \eqref{eq:altequiv}
and the statement about the bounded subsets
follow from \Cref{thm:BS13defjet} and the uniform definable $C^{m,\om}$ 
Whitney extension theorem: 

\begin{theorem}[{\cite{ParusinskiRainer:2023aa}}] \label{thm:defjetext}
    Let $0 \le m\le p$ be integers. Let $\om$ be a modulus of continuity.
    Let $X\subseteq \R^n$ be a definable closed set.
    Any definable bounded family of Whitney jets of class $C^{m,\om}$ on $X$ extends to a definable bounded family of $C^{m,\om}$-functions on $\R^n$ 
    which are of class $C^p$ outside $X$.
\end{theorem}

Clearly, boundedness is understood with respect to the natural norms.
The main theorem of \cite{ParusinskiRainer:2023aa} is actually more general.

Let us prove \eqref{eq:altequiv} and the statement about the bounded subsets.
It is obvious that 
$C^{1,\om}_{\on{def}}(\R^n)|_X \subseteq   \R^X_{\on{def}} \cap C^{1,\om}(\R^n)|_X$ and
\[
    \|f\|_{C^{1,\om}(\R^n)|_X} \le \|f\|_{C^{1,\om}_{\on{def}}(\R^n)|_X}, \quad f \in C^{1,\om}_{\on{def}}(\R^n)|_X.
\]    
Conversely, suppose that $f \in \R^X_{\on{def}} \cap C^{1,\om}(\R^n)|_X$. 
Then \Cref{thm:BS13defjet} implies that there is a bounded definable function $g : X \to \R^n$ 
such that $(f,g)$ is a definable Whitney jet of class $C^{1,\om}$ on $X$
with
\[
    \|(f,g)\|_{X,1,\om} \le C(n) \, \|f\|_{C^{1,\om}(\R^n)|_X}.
\]
By \Cref{thm:defjetext}, we conclude that $f \in C^{1,\om}_{\on{def}}(\R^n)|_X$ 
and that a subset of $\R^X_{\on{def}} \cap C^{1,\om}(\R^n)|_X$ which is bounded in  
$C^{1,\om}(\R^n)|_X$ is also bounded in $C^{1,\om}_{\on{def}}(\R^n)|_X$.

Let us now specialize to the case $\om(t) = t$ (we write $\|\cdot\|_{X,1,1} := \|\cdot\|_{X,1,\om}$ 
and $|\cdot|_{X,1,1} := |\cdot|_{X,1,\om}$ in this case) and prove \eqref{eq:equivnorms}. 
To this end, consider: 

\begin{proposition} \label{cor:C11}
    Let $X \subseteq \R^n$ be a definable compact set and let $f : X \to \R$ be a bounded definable function.
    Assume that $\cM_X$ carries the metric $\rh_\om$ with $\om(t) = t$.
    Then the following conditions are equivalent:
    \begin{enumerate}  
        \item $f$ is the restriction of a definable function $F \in C^{1,1}(\R^n)$.
        \item There exists a definable function $g : X \to \R^n$ such that $(f,g)$ is a Whitney jet of class $C^{1,1}$ 
            on $X$.
    \end{enumerate}
    Moreover, if $g$ is as in \thetag{2}, then there exists an extension $F$ of $f$ such that
    \begin{align} \label{eq:normbound}
        \|F\|_{C^{1,1}(\R^n)}  
        \le C(n)\, \|(f,g)\|_{X,1,1}.
    \end{align}       
\end{proposition}

\begin{proof}
    Let us recall a result of Azagra, Le Gruyer, and Mudarra \cite{Azagra:2018aa}:

    \emph{
        Let $(f,g)$ be a Whitney jet of class $C^{1,1}$ on a set $X \subseteq \R^n$ and let $M>0$ be such that 
        $|(f,g)|_{X,1,1}\le M$. 
        Then 
        \begin{align*}
            F &:= \on{conv}(h) - \frac{\sqrt n\, M}{2} \|\cdot \|_2^2, \quad{ where }
            \\
            h(x) &:= \inf_{y \in X} \Big( f(y) + \langle g(y), x-y \rangle + \frac{\sqrt n\, M}{2} \|x-y\|_2^2\Big) + \frac{\sqrt n\, M}{2} \|x\|_2^2, \quad x \in \R^n,
        \end{align*}
        and where $\on{conv}(h)$ is the convex envelope of $h$,
        defines a $C^1$-function $F : \R^n \to \R$ such that $F|_X = f$, $\nabla F|_X = g$, and $|\nabla F|_{\on{Lip}(\R^n,\R^n)} \le n\, M$. 
    }

    Note that the factor $\sqrt n$ appears, since we change from Euclidean to maximum norm. 
    The extension $F$ is optimal; cf.\ \cite[Theorem 3.4]{Azagra:2018aa}.
    The convex envelope $\on{conv}(h)$ 
    (i.e., the supremum of all convex, proper, l.s.c.\ functions $\vh \le h$) 
    can be expressed as
    \begin{equation*}
        \on{conv}(h)(x) = \inf \Big\{ \sum_{j=1}^{n+1} \la_j h(x_j) : 
        x = \sum_{j=1}^{n+1} \la_j x_j,\,   \sum_{j=1}^{n+1} \la_j =1,\, \la_j \ge 0 \Big\}.
    \end{equation*}
    or as $\on{conv}(h) = (h^*)^*$, where $h^*(x) := \sup_{y \in \R^n} \big(\langle y,x \rangle - h(y)\big)$ 
    is the convex conjugate.
    Now, if $X$ and $(f,g)$ are definable, we see that also the extension $F$ is definable 
    (because it is given by definable formulas). 
    After multiplication with a suitable definable $C^2$-cutoff function that equals $1$ in a small neighborhood of $X$
    and vanishes outside the $1$-neighborhood of $X$
    (here we use that $X$ is compact),
    we get that $F$ belongs to $C^{1,1}(\R^n)$.
    This shows that (2) implies (1). The opposite direction is clear.

    Now it is not hard to check \eqref{eq:normbound}. 
\end{proof} 

\begin{theorem} \label{thm:BS13defC11}
    Let $X \subseteq \R^n$ be a definable compact set and $f : X \to \R$ a definable function.
    Assume that
    the restriction $f|_Y$ to every subset $Y \subseteq X$ of cardinality $\# Y\le 3 \cdot 2^{n-1}$ has an 
    extension $F_Y\in C^{1,1}(\R^n)$ with $\|F_Y\|_{C^{1,1}(\R^n)} \le 1$.
    Then $f$ has a definable extension $F \in C^{1,1}(\R^n)$ such that 
    $\|F\|_{C^{1,1}(\R^n)} \le C(n)$. 
\end{theorem}

\begin{proof}
    Combine \Cref{thm:BS13defjet} with \Cref{cor:C11}. 
\end{proof}

Now we may prove \eqref{eq:equivnorms}.
Let $F \in C^{1,1}(\R^n)$ be an extension of $f$. 
Then \Cref{thm:BS13defC11} implies that there is a definable extension $G \in C^{1,1}(\R^n)$ of $f$ 
satisfying $\|G\|_{C^{1,1}(\R^n)} \le C(n) \, \|F\|_{C^{1,1}(\R^n)}$. 
Then \eqref{eq:equivnorms} follows easily.

\subsection{Definable Kirszbraun theorem} \label{ssec:Kirszbraun}
As a consequence of the result of \cite{Azagra:2018aa} used in the proof of \Cref{cor:C11}, 
a version of Kirszbraun's theorem on the extension of Lipschitz functions by an explicit formula is given in \cite[Theorem 1.2]{Azagra:2020aa}:

\emph{
    Let $X \subseteq \R^n$ be any set and $f : X \to \R^m$ a Lipschitz map with Lipschitz constant $M$, 
    where $\R^n$ and $\R^m$ carry the standard inner product (i.e., the Lipschitz constant $M$ is computed 
    with respect to the Euclidean norms, whence no factor $\sqrt n$).
    Then 
    \begin{align*}
        F(x) &:= \nabla_{\R^m} \on{conv}(g)(x,0), \quad x \in \R^n, \quad \text{ where}
        \\
        g(x,y) &:= \inf_{z \in X} \big( \langle f(z), y\rangle + \frac{M}2 \|x-z\|^2_2 \big) + \frac{M}{2} \|x\|_2^2 + M \|y\|_2^2,\quad (x,y) \in \R^n \times \R^m,
    \end{align*}
    defines a Lipschitz extension $F : \R^n \to \R^m$ of $f$ with the same Lipschitz constant $M$.
}

This follows easily by applying the mentioned theorem of \cite{Azagra:2018aa}
to the $1$-jet $(0,(0,f))$ on $X \times \{0\} \subseteq \R^n \times \R^m$.
Actually, the result is valid for maps between Hilbert spaces.

As a corollary we get a \emph{definable Kirszbraun theorem} (since all formulas are definable):

\begin{theorem}
    If $f : X \to \R^m$, $X \subseteq \R^n$, is a definable Lipschitz map, 
    then the map $F : \R^n \to \R^m$ defined above is a definable Lipschitz extension of $f$ 
    preserving the Lipschitz constant.    
\end{theorem}

The definable Kirszbraun theorem was first proved by Aschenbrenner and Fischer \cite{Aschenbrenner:2010aa}, 
but the explicit formula gives it immediately.

\begin{remark}
    Actually, the definable Kirszbraun theorem \cite[Theorem A]{Aschenbrenner:2010aa} holds in 
    any expansion $\fR =(R,0,1,+,\cdot,< ,\ldots)$ of a real closed ordered field that is \emph{definably complete} (i.e., 
    each non-empty definable subset of $R$ which is bounded from above has a least upper bound in $R$). 
    Definable completeness (which follows from o-minimality) is a necessary condition for the validity of the 
definable Kirszbraun theorem;
    see \cite[Proposition 5.2]{Aschenbrenner:2010aa}.
    A careful inspection of the proofs of \cite[Theorem 3.4]{Azagra:2018aa} and \cite[Theorem 1.2]{Azagra:2020aa}
    shows that the explicit extension formula given above still holds and thus 
    gives a short alternative proof in this general setting.
\end{remark}

\subsection{Remarks on \eqref{eq:equivnormsom}} \label{ssec:furtherremarks}

The obstacle for obtaining \eqref{eq:equivnormsom} for all definable moduli of continuity $\om$
is that we do not know if a bound of the type \eqref{eq:normbound} generally holds.
But, by a result of Paw{\l}ucki \cite[Theorem 1.2]{Pawlucki08aa}, 
Whitney jets of class $C^{m,\om}$ (not necessarily definable) on a definable closed set $X \subseteq \R^n$ 
extend to $C^{m,\om}$-functions and the extension is by a continuous linear operator 
which is a finite composite of operators that either preserve definability or 
are defined by integration with respect to a parameter.

Integration may lead out of the o-minimal structure one started with.
For instance, if one starts with globally subanalytic sets and maps, due to a result of Lion and Rolin \cite{Lion:1998aa}, 
one lands in the algebra of real functions generated by globally subanalytic functions 
and their logarithms.
For a class containing globally subanalytic functions and their complex exponentials (thus oscillatory functions)
that is stable under parameterized integration, see \cite{Cluckers:2018aa}.
Generally, for every o-minimal expansion $\sS$ of the real field there is an o-minimal expansion $\widetilde \sS$
in which the solutions of Pfaffian equations with $\sS$-definable $C^1$-coefficients are 
$\widetilde \sS$-definable, 
by Speissegger \cite{Speissegger:1999aa}.
We conjecture that 
\eqref{eq:equivnormsom} at least holds if the trace norm on the left-hand side is 
computed with respect to the functions definable in $\widetilde \sS$ instead of 
those definable in the structure $\sS$ we started with.

In \cite{Pawlucki08aa}, integration with respect to a parameter is used 
for smoothing operators that are linear, continuous, and preserve moduli of continuity.
In the proof of the definable Whitney extension theorem of class $C^m$ 
\cite{Kurdyka:1997ab, Kurdyka:2014aa, Thamrongthanyalak:2017aa} as well as of class $C^{m,\om}$ 
\cite{ParusinskiRainer:2023aa}, 
smooth cell decomposition 
and subtle inequalities for the derivatives of definable functions are used instead of smoothing operators.

\subsection*{Acknowledgements}
This research was funded in whole or in part by the Austrian Science Fund (FWF) DOI 10.55776/P32905.
A large part of the work on this paper has been done at Mathematisches Forschungsinstitut Oberwolfach 
(\emph{Oberwolfach Research Fellows (OWRF) ID 2244p, October 31 - November 19, 2022}).
We are grateful for the support and the excellent working conditions.


\def\cprime{$'$}
\providecommand{\bysame}{\leavevmode\hbox to3em{\hrulefill}\thinspace}
\providecommand{\MR}{\relax\ifhmode\unskip\space\fi MR }
\providecommand{\MRhref}[2]{%
  \href{http://www.ams.org/mathscinet-getitem?mr=#1}{#2}
}
\providecommand{\href}[2]{#2}

\end{document}